\documentclass[a4paper,12pt,final]{amsart}
\usepackage[utf8]{inputenc}
\usepackage[T1]{fontenc}
\usepackage[UKenglish]{babel}
\usepackage[a4paper,margin=28mm]{geometry}
\usepackage{verbatim}
\usepackage{mathtools}
\allowdisplaybreaks[4]
\usepackage{times}
\usepackage{dsfont,mathrsfs}
\usepackage{amsmath}
\usepackage{amsthm}
\usepackage{amssymb}
\usepackage{amsfonts}
\usepackage{latexsym}
\usepackage{booktabs}
\usepackage{graphicx}
\usepackage{epstopdf}

\usepackage[colorlinks,linkcolor=red,pagebackref]{hyperref}
\usepackage{xcolor}
\usepackage{tcolorbox}
\tcbuselibrary{skins, breakable}

\usepackage{ulem}
\usepackage{tikz}
\usetikzlibrary{positioning, shapes.geometric}

\makeatletter
\@namedef{subjclassname@2020}{\textup{2020} Mathematics Subject Classification}
\makeatother

\newtheorem{theorem}{Theorem}[section]
\newtheorem{lemma}[theorem]{Lemma}
\newtheorem{proposition}[theorem]{Proposition}

\newtheorem{assumption}[theorem]{Assumption}

\theoremstyle{definition}

\newtheorem{remark}[theorem]{Remark}
\numberwithin{equation}{section}
\renewcommand{\labelenumi}{\roman{enumi})}
\renewcommand\theenumi\labelenumi
\renewcommand{\leq}{\leqslant}
\renewcommand{\le}{\leqslant}
\renewcommand{\geq}{\geqslant}
\renewcommand{\ge}{\geqslant}

\newcommand{\Be}{\begin{equation}}
	\newcommand{\Ees}{\end{equation*}}
\newcommand{\Bes}{\begin{equation*}}
	\newcommand{\Ee}{\end{equation}}

\newcommand{\R}{\mathbb{R}}
\newcommand{\E}{\mathbb{E}}
\newcommand{\e}{\varepsilon}

\newcommand{\dif}{\mathrm{d}}

\usepackage{marginnote}
\begin{document}
	
	\title[Smoluchowski-Kramers limit of stable SDE]
	{Wasserstein bounds for invariant measures in the Smoluchowski-Kramers approximation with stable noise}

	\author[X. Jin]{Xinghu Jin}
	\address{Xinghu Jin: School of Mathematics,
		Hefei University of Technology,
		Hefei, Anhui, China;}
	\email{xinghujin@hfut.edu.cn}

	\author[W. Tao]{Wenxuan Tao*}
	\address{Wenxuan Tao (Corresponding Author): School of Mathematics, University of Birmingham, Birmingham, B15 2TT, United Kingdom}
	\email{wxt399@student.bham.ac.uk}

	\keywords{Smoluchowski-Kramers approxiamtion; $\alpha$-stable process; Wasserstein-1 distance; Stein's equation;}
	
	\subjclass[]{60B1;60G52;60G10}

	\begin{abstract}
	We study invariant position distributions in the Smoluchowski–Kramers approximation for Langevin equations driven by rotationally symmetric \(\alpha\)-stable noise, with \(\alpha\in(1,2)\). Under global Lipschitz continuity and dissipativity of the drift, we establish a Wasserstein-1 bound of order \(m^{1/\alpha}\) between the invariant position distribution of the underdamped system and the invariant distribution of its overdamped limit, where $m>0$ is the mass parameter. The proof combines a corrected-position test function with a regularized Stein equation. Cancellation of the transport and friction terms reduces the comparison to a first-moment estimate for the stationary velocity.
	\end{abstract}

	\maketitle
	
	\setcounter{tocdepth}{2}

	\section{Introduction} \label{Introduction}
	
	We consider the following Langevin system driven by a $d$-dimensional rotationally symmetric $\alpha$-stable L\'evy process, with $\alpha\in(1,2)$:
	\begin{align}\label{e:stableHamsys}
		\left\{
		\begin{array}{lll}
			\dif X_t^m &= Y_t^m \dif t, \\
			m\dif Y_t^m &= \left( b(X_t^m)-\gamma Y_t^m \right) \dif t + \sigma \dif Z_t,
		\end{array}
		\right.
	\end{align}
	where the constant $\gamma>0$ and  $\sigma\sigma^T$ is a strictly positive definite matrix, and $(Z_t)_{t\geq 0}$ is a  rotationally symmetric $\alpha$ stable process with L\'evy measure
	\begin{align}\label{levy measure}
		\nu(\dif z) &=
		\frac{\alpha 2^{\alpha-1} \pi^{-d/2}\Gamma((d+\alpha)/2)/\Gamma(1-\alpha/2)}{|z|^{\alpha+d}} \dif z
		=: \frac{d_{\alpha}}{|z|^{\alpha+d}} \dif z.
	\end{align}
	For each deterministic initial state $(x,y)\in\R^d\times\R^d$, we write
	$(X_t^{m,x,y},Y_t^{m,x,y})$ for the solution of \eqref{e:stableHamsys}
	with $(X_0^{m,x,y},Y_0^{m,x,y})=(x,y)$. 
	\begin{assumption}\label{Assumption b}
		The coefficient	$b:\R^d \to \R^d$ 
		is globally Lipschitz and there exists some constant $K>0$ such that for all $x,y\in \R^d$
		\begin{align*}
			|b(x)-b(y)| & \leq K |x-y|.
		\end{align*}
		Additionally, the coefficient $b$ satisfies the following dissipative condition, that is, there exist constants $c_1>0$ and $c_2>0$ such that for all $x,y\in \R^d$
		\begin{align*}
			\langle x-y, b(x)-b(y) \rangle & \leq -c_1 |x-y|^2+c_2.
		\end{align*}
	\end{assumption}

	We shall use the following SDE driven by stable process to approximate the law of $X_t^m$ as $m\to 0$:
	\begin{align}\label{X SDE}
		\dif X_t &= \gamma^{-1} b(X_t) \dif t + \gamma^{-1} \sigma  \dif Z_t,
	\end{align} 
	where the coefficient $b$ and noise process $Z$ admit the same meanings in \eqref{e:stableHamsys}. For the  process $X=(X_t)_{t\geq 0}$ in \eqref{X SDE}, $X_t^x$ denotes the solution starting from $X_0^x=x$; no initial velocity is prescribed for the limiting equation.
	
	Due to all limiting-results rely on  $m\to 0$, we only concern on 
		\begin{align}\label{E:small-mass-range}
			0<m\leq m_0
            := \min\left\{1, \frac{4c_1 \gamma^2}{16K^2+2c_1(1+2K)+c_1^2}
            \right\},
		\end{align}
	where constants $c_1,c_2$ and $K$ are all in Assumption \ref{Assumption b}.

	\subsection{Wasserstein-1 approximation for ergodic measure.}
	Under Assumption \ref{Assumption b}, \cite{Bao2022coupling} gave the exponential ergodicity for the process $(X_t^m,Y_t^m)_{t\geq 0}$ with ergodic measure $\mu_m$, denote $\mu_{m,1}$ the position marginal of the invariant law $\mu_m$. The corresponding limit SDE $(X_t)_{t\geq 0}$ in \eqref{X SDE} can also be proven ergodicity with ergodic measure $\mu$ by \cite[Theorem 1.3]{Liang2020Gradient} (see Proposition \ref{pro ergodic X}). 
	The small-mass comparisons below are considered for $0<m\leq m_0$,
	where $m_0$ is specified in \eqref{E:small-mass-range}.  Our main result is to investigate rate of convergence from ergodic measures $\mu$ to $\mu_{m,1}$ with respect to the Wasserstein-1 distance:
	\begin{align*}
		\mathcal{W}_{1}(\mu_{m,1},\mu)=\sup_{ h\in {\rm Lip}(1) }\{\mu_{m,1}(h)-\mu(h)\},
	\end{align*}
	where ${\rm Lip}(1)=\{h: |h(x)-h(y)|\leq |x-y|, \forall x, y \in \R^d\}$.

	\begin{theorem}\label{thm W1}
		Under Assumption \ref{Assumption b} and  $0<m\leq m_0$, let $\mu$ be the ergodic measure for $(X_t)_{t\geq 0}$ in SDE \eqref{X SDE} and $\mu_{m,1}$ be the position marginal of the invariant measure $\mu_m$  of system $(X_t^m,Y_t^m)_{t\geq 0}$ in \eqref{e:stableHamsys}. Then there exists $C>0$ independent of $m$ such that
		\begin{align*}
			\mathcal{W}_1(\mu_{m,1},\mu)
			&\leq Cm^{1/\alpha}.
		\end{align*}
	\end{theorem} 
	
	\begin{remark}
		Establishing ergodicity for system $(X_t^m,Y_t^m)_{t\geq 0}$ in \eqref{e:stableHamsys} for the multiplicative-noise case, with $\sigma=\sigma(x)$, appears to be more challenging. Although a Lyapunov construction similar to that in the next section may yield moment bounds under suitable assumptions on $\sigma$, further arguments are needed to establish exponential ergodicity. For additive noise, Bao and Wang obtained exponential ergodicity using a coupling approach~\cite{Bao2022coupling}. For multiplicative noise, Zhang established the existence of a smooth fundamental solution under a nonlocal H\"ormander-type condition\cite{Zhang2017fundamental}. Applying that result to the present kinetic system, however, requires a careful verification of its assumptions. In particular, the version of the result considered here may require a velocity-dependent coefficient $\sigma=\sigma(x,v)$ with suitable decay as $|v|\to\infty$, which would restrict its applicability to coefficients depending only on position. Moreover, such regularity does not by itself establish exponential ergodicity. Other condition like reachability is still needed.
	\end{remark}

	\subsection{Related works}
		
		The Smoluchowski-Kramers approximation has been extensively studied for stochastic systems driven by Brownian and general L\'evy noises.
		
		{\it Smoluchowski-Kramers approximation with Brownian noise:} The Smoluchowski-Kramers approximation describes the small-mass limit of Langevin-type stochastic differential equations driven by Brownian motion. In the classical setting, the stochastic forcing is continuous and Gaussian, with finite moments of all orders, and the limiting procedure reduces a second-order stochastic dynamics to an effective first-order equation as the mass tends to zero. The foundations of this theory can be traced back to Nelson \cite{Nelson1967dynamical}, while Freidlin \cite{Freidlin2004some} and Chen and Freidlin \cite{Chen2005Smoluchowski} investigated the small-mass limit and its implications for related probabilistic problems. Hottovy et al. \cite{Hottovy2015the} further studied the approximation for stochastic differential equations with state-dependent friction. The framework has also been used in statistical applications; in particular, He, Duan and Cheng \cite{He2019A} developed a parameter estimation method based on the Smoluchowski-Kramers approximation. Beyond the qualitative convergence of the small-mass system to its limiting dynamics, quantitative convergence rates have also been established in the Brownian setting. Van and Dung \cite{Van2020A} obtained Berry-Esseen bounds, providing explicit rates for the distributional approximation, while Son, Le and Duong \cite{Son2024rate} established convergence rates in total variation distance. These results show that the classical Brownian theory provides not only a rigorous foundation for the Smoluchowski-Kramers limit but also a quantitative framework for measuring the approximation error in different probability metrics. Liu, Liu and Xu \cite{Liu2025long} used the Stein method to study the Wasserstein-1 convergence rate for Smoluchowski-Kramers approximation.

		{\it Smoluchowski-Kramers approximation with L\'evy processes:} For general L\'evy processes, existing results often rely on finite-moment assumptions or finite jump activity. For example, Wang, Lv and Duan \cite{Wang2023Homogenization} established small-mass limits for stochastic Hamiltonian systems with multiplicative L\'evy noise under sufficiently high-order moment assumptions on the L\'evy measure, while Zhang \cite{Zhang2008Smoluchowski} considered non-Gaussian L\'evy noise under finite-moment or finite-intensity conditions. Recently, Zhao and Wang \cite{Zhao2025convergence} studied the Smoluchowski-Kramers approximation for Langevin equations driven by isotropic $\alpha$-stable noise with $\alpha\in(1,2)$. They obtained quantitative strong convergence for a suitable noise-scaling regime and established convergence in the Skorokhod topology in the critical case where uniform pathwise convergence generally fails. Related small-mass approximation results for stable-driven stochastic equations have also been obtained in \cite{Zhao2026on}. These works highlight the distinct mathematical features of Smoluchowski-Kramers approximations under stable noise and motivate further analysis based directly on the heavy-tailed structure of the stable L\'{e}vy measure.

		{\it Stein method and quantitative convergence of probability laws:} The Stein method is a powerful probabilistic framework for obtaining quantitative convergence rates between probability laws. Unlike classical weak-convergence arguments, which typically establish convergence without specifying an explicit rate, the Stein method provides a systematic approach to quantifying the discrepancy between a target distribution and its approximation under a suitable probability metric. Originally developed for normal approximation, it has since been extended to a broad range of probability distributions and stochastic systems; see, for example, Barbour and Chen \cite{Barbour2005an} and Chen, Goldstein and Shao \cite{Chen2011normal}. In particular, Stein's method has been successfully applied to stable-law approximations, including symmetric and asymmetric stable distributions and non-integrable stable limits \cite{Xu2019approximation,Chen2021stein,Chen2022non,Jin2020A}. It has also been used to quantify approximation errors for stochastic systems, including Euler-Maruyama schemes for diffusion processes and stable SDEs \cite{Fang2019multivariate,Chen2023Approximation,Chen2025Approximation}, as well as steady-state and invariant-measure approximations for queueing systems and their diffusion limits \cite{Gurvich2014diffusion,Jin2025An}. More broadly, the combination of Stein's method with semigroup estimates, Malliavin calculus, PDE techniques, coupling arguments, and generator expansions has provided an effective framework for quantitative distributional approximation in Markov processes, stochastic differential equations, and related numerical schemes. These developments highlight the versatility of Stein's method, particularly in settings involving non-Gaussian noise, stochastic dynamics, and invariant probability measures.

	\subsection{Methodology}
	Denote $\mathcal{A}_m$ the generator for process $(X_t^m,Y_t^m)_{t\geq 0}$ which is given in Lemma \ref{lemma Am},
	\begin{align}\label{generator Lm}
		\mathcal{A}_m f(x,y)
		&= \left\langle y, \frac{\partial}{\partial x} f(x,y) \right\rangle
		+ m^{-1} \left\langle b(x)-\gamma y, \frac{\partial}{\partial y} f(x,y) \right\rangle
		+(-\Delta)^{\alpha/2}_{ m^{-1}\sigma,2}f(x,y),
	\end{align}
	where $\frac{\partial}{\partial x} f$ and $\frac{\partial}{\partial y} f$ are partial derivatives with respected to $x$ and $y$ respectively, and 
	\begin{align*}
	(-\Delta)^{\alpha/2}_{A,2}f(x,y)
	&= \int_{\R_0^d}\left[f\left(x,y+A z\right) -f(x,y)- \langle A z, \frac{\partial}{\partial y} f(x,y) \rangle 1_{ \{ |Az|\leq 1 \} } \right] \frac{d_{\alpha}}{|z|^{\alpha+d}} \dif z,
	\end{align*}
	while $2$ in $(-\Delta)^{\alpha/2}_{A,2}$ means the second variable and $A$ is a matrix.
	
	It is well known that the generator $\mathcal{A}$ for the limit process $(X_t)_{t\geq 0}$ is 
	\begin{align}\label{generator L}
		\mathcal{A} f(x)&= \langle \gamma^{-1} b(x), \nabla f(x) \rangle +(-\Delta)^{\alpha/2}_{\gamma^{-1}\sigma} f(x),
	\end{align}
	where 
	\begin{align*}
	(-\Delta)^{\alpha/2}_{A} f(x)&= \int_{\mathbb{R}_0^d}\left[ f\left(x+Az\right)-f(x)- \langle Az, \nabla f(x) \rangle 1_{\{ |A z|\leq 1 \}} \right] \frac{d_{\alpha}}{|z|^{\alpha+d}} \dif z.
	\end{align*}
	
	Our approach combines Stein's method with a corrected position
	variable and a mollification argument. The correction cancels the
	singular transport and friction terms in the kinetic generator,
	whereas mollification justifies the generator calculation without
	requiring a bounded Hessian of the original Stein solution.

	For $h\in{\rm Lip}(1)$, We consider the following Stein equation
	\begin{align}\label{stein equ}
	\mathcal{A} f_h(x) &=h(x)-\mu(h),
	\end{align}
	where $\mathcal{A}$ and $\mu$ are generator and ergodic measure for process $(X_t)_{t\geq 0}$ in \eqref{X SDE}, additionally, it follows from Proposition \ref{pro ergodic X} that $\mu(h)$ is well defined for $h\in {\rm Lip}(1)$. Estimate for $\|\nabla f_h\|_\infty$ is shown in Lemma \ref{pro:stein regularies}, however, a bounded gradient alone does not ensure absolute convergence of
	the compensated small-jump integral defining $(-\Delta)^{\alpha/2}_{A,2}f_h$ 
	when $\alpha>1$. We therefore regularize the Stein solution before using it as a test function for the kinetic system.
	
	Let $\rho\in C_c^\infty(\mathbb R^d)$ be nonnegative, supported in
	the unit ball, and satisfy $\int_{\R^d} \rho(z)\dif z=1$. For
	$\varepsilon>0$, define
	\begin{align*}
	\rho_\varepsilon(z)=\varepsilon^{-d}\rho(z/\varepsilon),
	\qquad
	f_{h,\varepsilon}=f_h*\rho_\varepsilon,
	\qquad
	h_\varepsilon=h*\rho_\varepsilon.
	\end{align*}
	
	For $h\in {\rm Lip}(1)$, convolving the Stein's equation \eqref{stein equ} with $\rho_\varepsilon$, we obtain
	\begin{align*}
		\gamma^{-1}(b\cdot\nabla f_h)*\rho_\varepsilon
		+((-\Delta)^{\alpha/2}_{A} f_h)*\rho_\varepsilon
		&=h_\varepsilon-\mu(h).
	\end{align*}
	Since convolution preserves constants, and using translation invariance of $(-\Delta)^{\alpha/2}_{A}$, that is, 
	\begin{align*}
		((-\Delta)^{\alpha/2}_{A}f_h)*\rho_\varepsilon
		&=(-\Delta)^{\alpha/2}_{A} (f_h*\rho_\varepsilon)
		=(-\Delta)^{\alpha/2}_{A} f_{h,\varepsilon},
	\end{align*}
	we obtain
	\begin{align*}
		(-\Delta)^{\alpha/2}_{A} f_{h,\varepsilon}
		&=h_\varepsilon-\mu(h)
		-\gamma^{-1}(b\cdot\nabla f_h)*\rho_\varepsilon.
	\end{align*}	
	
	Due to
	\begin{align*}
		\langle b(x),\nabla f_{h,\varepsilon}(x)\rangle
		&=\left\langle
		b(x),\int_{\R^d} \nabla f_h(x-z)\rho_\varepsilon(z) \dif z
		\right\rangle
		=\int_{\R^d} 
		\langle b(x),\nabla f_h(x-z)\rangle\rho_\varepsilon(z)\dif z,
	\end{align*}
	one knows that
	\begin{align}\label{e:newStein}
	\mathcal{A}f_{h,\varepsilon}(x)
	&=
	h_\varepsilon(x)-\mu(h)+\mathcal{R}_{h,\varepsilon}(x),
	\end{align}
	where
	\begin{align*}
	\mathcal{R}_{h,\varepsilon}(x)
	=
	\gamma^{-1}\int_{\mathbb R^d}
	\left\langle
	b(x)-b(x-z),\nabla f_h(x-z)
	\right\rangle
	\rho_\varepsilon(z)\dif z.
	\end{align*}
	Related regularities for above approximation are given in Proposition \ref{lem:newReg}.
	
	We introduce the following refined test function:
	\begin{align*}
		\Phi_{m,h,\e}(x,y)
		&=f_{h,\e}\left(\omega\right)
		=f_{h,\e}\left(x+\gamma^{-1}my\right)
	\end{align*}
	with
	\begin{align}\label{e:omega}
		\omega &= x+\gamma^{-1}my.
	\end{align}
	
	For each fixed $\varepsilon>0$, the term $(-\Delta)^{\alpha/2}_{A,2} \Phi_{m,h,\e}(x,y)$ is well defined. Indeed, Taylor's formula bounds its compensated
	small-jump integrand by
	$C\|\nabla^2f_{h,\varepsilon}\|_\infty|z|^2$, while the
	large-jump integrand is bounded by
	$C\|\nabla f_{h,\varepsilon}\|_\infty|z|$. 
	
	The corrected position $x+\gamma^{-1}my$ in $\Phi_{m,h,\e}$ makes the transport term and
	the velocity-friction term in $\mathcal{A}_m\Phi_{m,h,\e}$ cancel. The remaining expression has the same drift and noise structure as the overdamped
	generator, evaluated at the corrected position. The discrepancy can
	therefore be estimated through the displacement $\gamma^{-1}my$ and
	the regularity of $f_{h,\e}$, rather than by estimating the velocity term alone.
	
	With this choice, we need to establish the corresponding generator comparison
	\begin{align*}
		\mathcal{A}_m\Phi_{m,h,\e}(x,y)-\mathcal{A}f_{h,\e}(\omega),
	\end{align*}
	which provides the key estimate for deriving the Wasserstein-1 convergence rate. This approach differs from the one employed by Liu, Liu and Xu \cite{Liu2025long}, and offers an alternative way to exploit the structure of the Smoluchowski--Kramers approximation. In fact, the same generator-comparison argument can also be adapted to obtain the Wasserstein-1 convergence rate considered in \cite{Liu2025long}.
	
	For $\alpha$-stable noise with $\alpha\in(1,2)$, the refined test function $\Phi_{m,h,\e}$ is particularly effective in the case of a constant friction coefficient $\gamma>0$. Extending this approach to state-dependent friction coefficients $\gamma=\gamma(x)$, which would require a corresponding modification of the test function and a more delicate generator analysis, is an interesting direction for our future research.
	
	After getting the generator comparison in Proposition \ref{lem:generator comparison}, it reminders to get the moment estimates, ergodicity arguments, Stein equations
	and regularity estimates which are in Sections~\ref{org:generators}-\ref{org:stein} respectively. The proof of the main theorem is given in Section~\ref{sec proof}.
	
	\subsection{Notation} In expressions involving $\E_{\mu_m}$, this abbreviation
	refers instead to a stationary initial pair with law $\mu_m$. For a random variable $W$, denote $\mathcal{L}(W)$ the law of $W$. $W_1 \overset{d}{=} W_2$ means $W_1$ and $W_2$ admit the same law.

	The constants $c$ and $C$ may take different values in different places; unless otherwise specified, they will always be independent of $m$. $A\lesssim B$ means that $A\leq cB$ for some constant $c>0$ independent of $m$.

    For function $f:\R^d \to \R$, the uniform norm reads $\| f \|_{\infty}=\sup_{x\in \R^d} |f(x)|$, while the $L^1$ norm is given by $\| f \|_{L^1}=\int_{\R^d} |f(x)| \dif x$.


	\section{Proof of Main result}\label{sec proof} 
	
	Regularities for modified Stein's equation \eqref{e:newStein} by convolution are below which are proven in Section \ref{org:stein}.
	\begin{proposition}\label{lem:newReg}
	Under Assumption~\ref{Assumption b} and for $h\in {\rm Lip}(1)$ and $\e\in (0,1)$, the following estimates hold with a constant independent of $h$ and $\e$
	\begin{align*}
		\| \nabla f_{h,\e} \|_{\infty} \leq C,
		\quad 
		\| \nabla^2 f_{h,\e} \|_{\infty} \leq C\e^{-1},
		\quad 
		\| h-h_{\e}\|_{\infty} \leq C\e,
		\quad 
		\| \mathcal{R}_{h,\varepsilon} \|_{\infty} \leq C\e.
		\end{align*}
		Additionally,  $\mathcal{A}_m\Phi_{m,h,\e}(x,y)$ is well defined and, for every $t\ge 0$,
		\begin{align}\label{e:EAmPhi}
		\E_{\mu_m} \mathcal{A}_m\Phi_{m,h,\e}(X_t^m,Y_t^m) &=0.
		\end{align} 
	\end{proposition}

	The following generators comparison plays a crucial role in proving the weak convergence rate which is proven in Section \ref{org:generators}.
	\begin{proposition}\label{lem:generator comparison}
		Under Assumption \ref{Assumption b}, let $f\in \mathcal{C}^2(\R^d)$ have bounded gradient and Hessian, and define $\Phi_m(x,y)=f(x+\gamma^{-1}my)$, one knows that 
		\begin{align*}
		\mathcal{A}_m \Phi_{m}(x,y) &= \mathcal{A} f(\omega)+\gamma^{-1} \langle b(x)-b(\omega),\nabla f(\omega) \rangle.
		\end{align*}
		Additionally, there exists $C>0$ depending only on $\gamma$ and Lipschitz constant of $b$ such that
		\begin{align*}
		|\mathcal{A}_m \Phi_m(x,y)-\mathcal{A}f(\omega)| 
		&\leq Cm\| \nabla f \|_{\infty} |y|,
		\end{align*}
		where $\omega$ is defined in \eqref{e:omega}.
	\end{proposition}
	
	We can obtain the following moment estimates which are proven in Section \ref{org:moments}.
	\begin{lemma}\label{lem:moment estimates}
		Under Assumption~\ref{Assumption b}, and $m_0$ in \eqref{E:small-mass-range}, such that, for every $\beta\in(0,\alpha)$,  $0<m\leq m_0$ and $t\geq0$, there exists some $C>0$ independent of $m$ such that
		\begin{align}
			\E_{\mu_m}|X_t^{m}|^\beta
			&\leq C, \label{Xm moment estimate} \\
			\E_{\mu_m}|Y_t^{m}|^\beta
			&\leq Cm^{\beta(1/\alpha-1)}. \label{Ym moment estimate}
		\end{align} 
	\end{lemma}

	\begin{proof}[Proof of Theorem \ref{thm W1}.]
    Since $\mu_{m,1}$ is the marginal distribution (with respect to the position variable, $X_t^m$) of the invariant measure of the process $(X_t^m,Y_t^m)_{t\geq 0}$ whose ergodic measure is denoted by $\mu_m$, then for any Borel measure set $A\in \mathcal{B}(\R^d)$, one knows that 
		\begin{align*}
			\mu_{m,1}(A) &= \mu_m(A \times \R^d),
		\end{align*}
		which imples that for any
		$f:\R^d \to \R$ satisfing $\mu_{m,1}(|f|)<\infty$, we know
		\begin{align}\label{marginal equal}
			\mu_{m,1}(f) &= \int_{\R^d} f(x) \mu_{m,1}(\dif x)
			= \int_{\R^d \times \R^d} f(x) \mu_m(\dif x, \dif y)
			=\mu_m(f).
		\end{align}

		We aim to get the upper bounded for $\mathcal{W}_1(\mu_{m,1},\mu)$. For the process $(X_t^m,Y_t^m)_{t\geq 0}$, define the corrected position process by
		\begin{align*}
			\Omega_t^m &=X_t^m+\gamma^{-1}mY_t^m.
		\end{align*}
		It follows from \eqref{marginal equal} that
		\begin{align}\label{e:W1}
		\mathcal{W}_1(\mu_{m,1},\mu)
		&= \sup_{h\in {\rm Lip}(1)} \{ \mu_{m,1}(h)-\mu(h) \}
		= \sup_{h\in {\rm Lip}(1)} 
		\{ 
		\E_{\mu_{m,1}}[h(X_t^m)- \mu(h)] 
		\} \nonumber \\
		&= \sup_{h\in {\rm Lip}(1)} 
		\{ 
		\E_{\mu_{m}}[h(X_t^m)- \mu(h)] 
		\}  \nonumber \\
		&\leq \sup_{h\in {\rm Lip}(1)} 
		\{ 
		\E_{\mu_{m}}[h(X_t^m)-h_{\e}(\Omega_t^m)] \}
		+ \sup_{h\in {\rm Lip}(1)} \{ \E_{\mu_{m}} [	h_{\e}(\Omega_t^m)-\mu(h)]
		\}
		=:{\rm I} + {\rm II}.
		\end{align}
		We need to estimate ${\rm I}$ and ${\rm II}$ sperately. 
		
		For ${\rm I}$, by $\| h-h_{\e}\|_{\infty} \lesssim \e$ in Proposition \ref{lem:newReg} and $h\in {\rm Lip}(1)$, one knows that
		\begin{align}\label{e:I}
		{\rm I} &= \sup_{h\in {\rm Lip}(1)} 
		\{ \E_{\mu_{m}}[h(X_t^m)-h_{\e}(\Omega_t^m)] \} \nonumber  \\
		&=\sup_{h\in {\rm Lip}(1)} 
		\{ \E_{\mu_{m}}[h(X_t^m)-h(\Omega_t^m)+h(\Omega_t^m)-h_{\e}(\Omega_t^m)] \}  \nonumber \\
		&\lesssim \E_{\mu_{m}}[|X_t^m-\Omega_t^m|+\e] 
		\lesssim 
		m\E_{\mu_{m}}|Y_t^m|+\e
		\lesssim m^{1/\alpha} + \e,
		\end{align}
		where the last inequality holds from \eqref{Ym moment estimate} with $\beta=1$.
		
		For {\rm II}, it follows from \eqref{e:newStein} and Proposition \ref{lem:generator comparison} with $\Phi_m=\Phi_{m,h,\e}$ and $f=f_{h,\e}$ that
		\begin{align*}
		\mathcal{A}_m \Phi_{m,h,\e}(x,y) &= \mathcal{A} f_{h,\e}(\omega)+\gamma^{-1} \langle b(x)-b(\omega),\nabla f_{h,\e}(\omega) \rangle \\
		&=h_\varepsilon(\omega)-\mu(h)+\mathcal{R}_{h,\varepsilon}(\omega)+\gamma^{-1} \langle b(x)-b(\omega),\nabla f_{h,\e}(\omega) \rangle.
		\end{align*}
		Combining this with \eqref{e:EAmPhi} in	Proposition \ref{lem:newReg}, one knows that
		\begin{align*}
		0&=\mathbb{E}_{\mu_m}\mathcal{A}_m \Phi_{m,h,\e}(X_t^m,Y_t^m) \\ 
		&=\mathbb{E}_{\mu_m} [h_\varepsilon(\Omega_t^m)-\mu(h)]
		+\mathbb{E}_{\mu_m} \mathcal{R}_{h,\varepsilon}(\Omega_t^m)+\gamma^{-1} \mathbb{E}_{\mu_m} [\langle b(X_t^m)-b(\Omega_t^m),\nabla f_{h,\e}(\Omega_t^m) \rangle].
		\end{align*}	
		It implies that
		\begin{align}\label{e:II}
		{\rm II} &= \sup_{h\in {\rm Lip}(1)} 
		\{ 
		\E_{\mu_{m}}[h_{\e}(\Omega_t^m)-\mu(h)]
		\} \nonumber \\
		&=\sup_{h\in {\rm Lip}(1)} \{-\mathbb{E}_{\mu_m} \mathcal{R}_{h,\varepsilon}(\Omega_t^m)-\gamma^{-1} \mathbb{E}_{\mu_m} [\langle b(X_t^m)-b(\Omega_t^m),\nabla f_{h,\e}(\Omega_t^m) \rangle] \} \nonumber \\
		&\lesssim \e+ \| \nabla f_{\e,h} \|_{\infty} \mathbb{E}_{\mu_m}|X_t^m-\Omega_t^m|
		\lesssim \e + m^{1/\alpha},
		\end{align}	
		where the last inequality holds from Proposition \ref{lem:newReg} and Lemma  \ref{lem:moment estimates}.
		
		It follows from \eqref{e:W1}, \eqref{e:I} and \eqref{e:II} that taking $\e=m^{1/\alpha}$, 
		\begin{align*}
		\mathcal{W}_1(\mu_{m,1},\mu)
		&\lesssim \e+m^{1/\alpha}
		\lesssim m^{1/\alpha}.
		\end{align*}
		The proof of Theorem~\ref{thm W1} is complete. 
	\end{proof}
	
	\section{Generator comparisons}\label{org:generators}
	\subsection{The generator of the Langevin system}
	Recall that the generator $\mathcal{A}_m$ for process $(X_t^m,Y_t^m)_{t\geq 0}$, which is defined as for $f: \R^d \times \R^d \to \R$ belonging to the domain of $\mathcal{A}_m$, denoted by  $\mathcal{D}(\mathcal{A}_m)$:
	\begin{align*}
		\mathcal{A}_m f(x,y) &=\lim_{t\to 0}\frac{1}{t}(\E f(X_t^{m,x,y},Y_t^{m,x,y})-f(x,y)).
	\end{align*}
	\begin{lemma}\label{lemma Am}
		The expression for $\mathcal{A}_m$ is given in \eqref{generator Lm}. 
	\end{lemma}
	
	\begin{proof}[Proof of Lemma \ref{lemma Am}]
	For $f\in \mathcal{D}(\mathcal{A}_m)$, one knows that
			\begin{align*}
				\E f(X_t^{m,x,y},Y_t^{m,x,y})-f(x,y)
				&= \E f(X_t^{m,x,y},Y_t^{m,x,y})-\E f(x,Y_t^{m,x,y}) \\
				&+\E f(x,Y_t^{m,x,y}) - f(x,y).
			\end{align*}
			
			Recall that
			\begin{align*}
				\dif X_t^m &= Y_t^m \dif t,
			\end{align*}
			then
			\begin{align*}
				X_t^{m,x,y} &= x+ \int_0^t Y_s^{m,x,y} \dif s,
			\end{align*}
			by the Taylor's explansion, one gets that
			\begin{align*}
				&\E f(X_t^{m,x,y},Y_t^{m,x,y})-\E f(x,Y_t^{m,x,y}) \\
				&= \E \int_0^1  \langle \nabla_{x+r[X_t^{m,x,y}-x]} f(x+r[X_t^{m,x,y}-x],Y_t^{m,x,y}),X_t^{m,x,y}-x \rangle \dif r,
			\end{align*}
			which implies 
			\begin{align*}
				\lim_{t\to 0} \frac{1}{t}\{ \E f(X_t^{m,x,y},Y_t^{m,x,y})-\E f(x,Y_t^{m,x,y}) \}
				&= \langle y, \frac{\partial}{\partial x} f(x,y) \rangle.
			\end{align*}
			
			Due to
			\begin{align*}
				\dif Y_t^m &= m^{-1}\left( b(X_t^m)-\gamma Y_t^m \right) \dif t + m^{-1}\sigma \dif Z_t,
			\end{align*}
			by It\^{o}'s formula (see \cite[Theorem 4.4.7]{Applebaum2009levy}), one knows that
			\begin{align*}
				\lim_{t\to 0} \frac{1}{t}\{ \E f(x,Y_t^{m,x,y}) - f(x,y) \}
				&= m^{-1} \left\langle b(x)-\gamma y, \frac{\partial}{\partial y} f(x,y) \right\rangle
				+(-\Delta)^{\alpha/2}_{ m^{-1}\sigma,2}f(x,y).
			\end{align*}
			Thus, we can get the corresponding expression in \eqref{generator Lm}.
	\end{proof}

	\begin{proof}[Proof of Proposition  \ref{lem:generator comparison}] 
		For any function $f\in \mathcal{C}^2(\R^d,\R)$, denoting 
		\begin{align*} 
		\Phi_m(x,y) = f\left(x+\gamma^{-1} my \right)  =  f(\omega),
		\end{align*} 
		one knows that 
		\begin{align*}
			\frac{\partial }{\partial x} \Phi_m(x,y) = \nabla f(\omega),
			\quad 
			\frac{\partial }{\partial y} \Phi_m(x,y) = \gamma^{-1} m \nabla f(\omega) .
		\end{align*}
		
		(i) The first two parts of $\mathcal{A}_m \Phi_m(x,y)$.
		\begin{align*}
			& m^{-1} \langle b(x)-\gamma y, \frac{\partial }{\partial y} \Phi_m(x,y) \rangle
			+\langle y, \frac{\partial }{\partial x} \Phi_m(x,y) \rangle \\
			&=m^{-1} \langle b(x)-\gamma y, 
			m \gamma^{-1} \nabla f(\omega)  \rangle 
			+\langle y, \nabla f(\omega) \rangle \\
			&=\langle \gamma^{-1} b(x),  \nabla f(\omega) \rangle.
		\end{align*}

		(ii) The non-local part of $\mathcal{A}_m \Phi_m(x,y)$. Due to 
		\begin{align*}
			\Phi_m\left(x,y+m^{-1}\sigma z  \right) - \Phi_m(x,y)
			&=f\left( x+m\gamma^{-1}(y+m^{-1}\sigma z) \right)-f\left( x+\gamma^{-1} my \right) \\
			&=f\left(\omega+\gamma^{-1}\sigma z\right)-f(\omega),
		\end{align*}
		and 
		\begin{align*}
			\langle \frac{\partial}{\partial y} \Phi_m(x,y), m^{-1} \sigma z \rangle
			&=\langle \nabla f(\omega), \gamma^{-1} \sigma z \rangle,
		\end{align*}
		one knows that
		\begin{align*}
			&(-\Delta)^{\alpha/2}_{m^{-1}\sigma,2} \Phi_m(x,y) \\
			&= \int_{\R_0^d}\left[\Phi_m\left(x,y+m^{-1}\sigma z\right) -\Phi_m(x,y)-\langle m^{-1}\sigma z, \frac{\partial}{\partial y} \Phi_m(x,y) \rangle 1_{ \{ |m^{-1}\sigma z|\leq 1 \} } \right] \frac{d_{\alpha}}{|z|^{\alpha+d}} \dif z \\
			&=\int_{\R_0^d}
			\left[
			f\left(\omega+\gamma^{-1}\sigma z\right)-f(\omega)-\langle \nabla f(\omega),\gamma^{-1}\sigma z \rangle 1_{\{ |m^{-1}\sigma z|\leq 1 \} } \right] \frac{d_{\alpha}}{|z|^{\alpha+d}} \dif z \\
			&=\int_{\R_0^d}
			\left[
			f\left(\omega+\gamma^{-1} \sigma z\right)-f(\omega)-\langle \nabla f(\omega), \gamma^{-1}\sigma z \rangle 1_{\{ |\gamma^{-1} \sigma z|\leq 1 \} } \right] \frac{d_{\alpha}}{|z|^{\alpha+d}} \dif z \\
			&=(-\Delta)^{\alpha/2}_{\gamma^{-1} \sigma }f(\omega).
		\end{align*}
		Indeed, the third equality holds from below: for any fix $m>0$, by the symmetry of the L\'{e}vy measure and the oddness of the linear term, the truncation function $1_{ \{ |m^{-1}\sigma z|\leq 1 \} }$ can be replaced by $1_{ \{ |\gamma^{-1}\sigma z| \leq 1  \}  }$, that is,
		\begin{align*}
			\int_{\R_0^d} \langle \nabla_{\omega} f(\omega),\gamma^{-1}\sigma z \rangle \left[ 1_{ \{ |\gamma^{-1}\sigma z| \leq 1  \}  } - 1_{ \{ |m^{-1}\sigma z|\leq 1 \} }    \right] \frac{d_{\alpha}}{|z|^{\alpha+d}} \dif z =0. 
		\end{align*}
		Combining above two steps, one knows that
		\begin{align*}
			\mathcal{A}_m \Phi_m(x,y)
			&=\langle \gamma^{-1} b(x),   \nabla f(\omega) \rangle +(-\Delta)^{\alpha/2}_{\gamma^{-1} \sigma }f(\omega) \\
			&=\mathcal{A}f(\omega)
			+ \langle \gamma^{-1} (b(x)-b(\omega)), \nabla f(\omega) \rangle.
		\end{align*}
		Combining this with $\| \nabla f \|_{\infty}<\infty$, Lipschitz continuous for $b$ and \eqref{e:omega}, we can get
		\begin{align*}
		|\mathcal{A}_m \Phi_m(x,y)
		-\mathcal{A}f(\omega)|
		&\leq C\| \nabla f \|_{\infty} |x-\omega|	
		\leq Cm\| \nabla f \|_{\infty} |y|.
		\end{align*}
		The proof of Proposition \ref{lem:generator comparison} is complete.
	\end{proof}

	\section{Ergodicity}\label{org:ergodicity}
	\subsection{The underdamped Langevin system}
	The exponential ergodicity for process $(X_t^m,Y_t^m)_{t\geq 0}$ can be obtained from \cite[Theorem 1.1]{Bao2022coupling}.
	\begin{proposition}\label{org:kinetic-ergodicity}
		For $0<m\leq m_0$, with $m_0$ defined in \eqref{E:small-mass-range}, and under Assumption \ref{Assumption b}, the process $(X_t^m,Y_t^m)_{t\geq 0}$ in \eqref{e:stableHamsys} are exponential ergodicity with ergodic measures $\mu_m$.
	\end{proposition}
	
	\begin{proof}[Proof of Proposition~\ref{org:kinetic-ergodicity}]
	The exponential ergodicity for process $(X_t^m,Y_t^m)_{t\geq 0}$ can be obtained from \cite[Theorem 1.1]{Bao2022coupling} by verifying the following conditions.

		(i) Under global Lipschitz continuous condition for $b$ in Assumption \ref{Assumption b}, we also know that the coefficient $b(x)-\gamma y$ is also global Lipschitz continuous, which satisfies the conditon $(A_0)$ in \cite[Theorem 1.1]{Bao2022coupling} holds.
		
		(ii) Under the dissipative condition for $b$ in Assumption \ref{Assumption b}, the Lyapunon condition holds in \eqref{E:Lyapunov condition} for the Lyapunov function $V_m$ in \eqref{E:underdamped-Lyapunov}, which implies the conditon $(A_1)$ in \cite[Theorem 1.1]{Bao2022coupling}.
		
		(iii) It follows from Lemmas 4.1 and 4.2 from~\cite{Bao2022coupling} that the L\'{e}vy measure $\nu$ in \eqref{levy measure} satisfies the conditon $(A_2)$ in \cite[Theorem 1.1]{Bao2022coupling}.
	\end{proof}

	\subsection{The overdamped limiting system}
	We can also get the exponential ergodicity for limit SDE $(X_t)_{t\geq 0}$ in \eqref{X SDE}, that is,
	\begin{proposition}\label{pro ergodic X}
		Under Assumption \ref{Assumption b}, the limit SDE $(X_t)_{t\geq 0}$ in \eqref{X SDE} is exponential ergodicity with ergodic measure $\mu$: There exist constants $C,c>0$ such that 
		\begin{align*}
			\mathcal{W}_1(\mathcal{L}(X_t^x),\mu) & \leq C e^{-ct}(1+|x|).
		\end{align*}
		Additionally, $\mu(|x|)<\infty$.
	\end{proposition}

	\begin{proof}[Proof of Proposition~\ref{pro ergodic X}]
	(i) The exponenetial ergodicity for process $(X_t)_{t\geq 0}$ can be obtained by \cite[Theorem 1.3]{Liang2020Gradient} (or \cite[Theorem 1.2]{Wang2016Lp}, or 
	\cite[Lemma 2.2]{Jin2026Optimal})  under the dissipative condition in Assumotion \ref{Assumption b}, that is, ther exist some strictly positive constants $c,C$ such that
	\begin{align*}
	\mathcal{W}_1(\mathcal{L}(X_t^x),\mathcal{L}(X_t^y)) &\leq Ce^{-c t} |x-y|,\quad \forall x,y \in \R^d, t>0,
	\end{align*}
	which also implies the gradient estimate for semigroup $P_t$ of process $(X_t)_{t\geq 0}$, that is, for Lipschitz continuous function $h$, one knows that 
	\begin{align*}
	|\nabla P_t h(x)| &\leq C e^{-c t}. \quad \forall x \in \R^d, t>0.
	\end{align*}
			
	(ii) The moment estimate for ergodic measure $\mu$ can be obtained by Lyapunov condition. 	
	For the stable SDE~\eqref{X SDE}, let
	\begin{equation*}
	V(x)=(1+|x|^2)^{1/2}.
	\end{equation*}
	Using the dissipative estimate \eqref{e:drift-coercivity} shown in the next section, with $a$ and $B$ defined there, we obtain
	\begin{align*}
	\gamma^{-1}\langle b(x),\nabla V(x)\rangle
	&\leq\frac{1}{\gamma}(1+|x|^2)^{-1/2}(-a|x|^2+B)
	\leq-\frac{ a}{\gamma}V(x)+\frac{a+B}{\gamma}.
	\end{align*}
	The Hessian and second-difference argument used in
	\eqref{E:SecondDiff_W} applies equally to $V$ on $\R^d$ with the choice $\beta=1$.
			By symmetry and with $A=\gamma^{-1}\sigma$, its jump contribution is
			\[
			\frac12\int_{\R_0^d}[V(x+Az)+V(x-Az)-2V(x)]\nu(\dif z),
			\]
			whose absolute value is bounded by
			\[
			C_1\left(\int_{|z|\leq1}|z|^2\nu(\dif z)
			+\int_{|z|>1}|z|\nu(\dif z)\right)\leq  C_1.
			\]
			Hence there exist constants $\widetilde c_1>0,C_1>0$ such that
			\begin{align*}
				\mathcal{A}V(x) &\leq-\widetilde c_1 V(x)+C_1.
			\end{align*}

			Using It\^{o}'s formula and taking expectation with respective to ergodic measure $\mu$, that is, $X_0 \sim \mu$, by stationrity 
			\begin{align*}
				\E_{\mu} V(X_t) &= \mu(V), \forall t\geq 0. 
			\end{align*}
			Thus, one knows that
			\begin{align*}
				\E_{\mu} V(X_t)
				\leq e^{-\widetilde{c}_1 t} \E_{\mu} V(x)
				+\frac{C_1}{\widetilde{c}_1}(1-e^{-\widetilde{c}_1 t}),
			\end{align*}
			which implies that
			\begin{align*}
				\mu(V) &\leq C.
			\end{align*}
			Combining with the definition of $V$, one can get the desired estimate.
	\end{proof}

	\section{Moment estimates}\label{org:moments}

	Throughout this section, $\beta\in(0,\alpha)$ denotes a moment order,
	whereas $p=\gamma^{-1}my\in\R^d$ denotes the scaled momentum. Constants may depend on $\beta,\alpha,d,\gamma,\sigma$ and
	the coefficients in Assumption~\ref{Assumption b}, but not on $m,t,x,y$.
	
	\begin{proof}[Proof of Lemma \ref{lem:moment estimates}]
		
		We first use a Lyapunov function to estimate the position moments.
		We then apply variation of constants to obtain the sharper velocity
		estimate, keeping its initial layer explicitly.
		
		\underline{(i) Position moments.}
		Recall that the generator of the Langevin system is
		\begin{align*}
			&\mathcal A_m f(x,y)
			=m^{-1}\langle b(x)-\gamma y,\frac{\partial}{\partial y}  f(x,y)\rangle
			+\langle y,\frac{\partial}{\partial x} f(x,y)\rangle\\
			&\qquad \qquad  +\int_{\R_0^d}\Bigl[
			f(x,y+m^{-1}\sigma z)-f(x,y)
			-\langle m^{-1}\sigma z,\frac{\partial}{\partial y} f(x,y)\rangle
			1_{\{|m^{-1}\sigma z|\leq1\}}
			\Bigr]\nu(\dif z).
		\end{align*}
		Denote by $\omega=x+\gamma^{-1}my$ and $p=\gamma^{-1}my$
		the corrected position and scaled momentum. We set
		\begin{equation}\label{E:underdamped-Lyapunov}
			\begin{aligned}
				H_m(x,y)&=\left|x+\gamma^{-1}my\right|^2
				+\left|\gamma^{-1}my\right|^2=|\omega|^2+|p|^2,\\
				V_m(x,y)&=(1+H_m(x,y))^{\beta/2}.
			\end{aligned}
		\end{equation}
		Since $x=\omega-p$ and $my=\gamma p$, we have
		\begin{align*}
			|x|^2+m^2|y|^2
			&=|\omega-p|^2+\gamma^2|p|^2
			\leq(2+\gamma^2)H_m(x,y),\\
			H_m(x,y)&\leq2|x|^2+3\gamma^{-2}m^2|y|^2
			\leq C(|x|^2+m^2|y|^2).
		\end{align*}
		The comparison constants are independent of $m$.
		
		For the solution starting from $(x,y)$, define
		\[
		\Omega_t^m=X_t^{m}+\gamma^{-1}mY_t^{m},\qquad
		p_t^m=\gamma^{-1}mY_t^{m},\qquad u_t^m=(\Omega_t^m,p_t^m).
		\]
		Then $u_0^m=(x+\gamma^{-1}my,\gamma^{-1}my)$, and
		\begin{align}\label{E:ScaledSys}
			\begin{cases}
				\dif\Omega_t^m=\gamma^{-1}b(\Omega_t^m-p_t^m)\dif t
				+\gamma^{-1}\sigma\dif Z_t,\\
				\dif p_t^m=\bigl(\gamma^{-1}b(\Omega_t^m-p_t^m)-\gamma m^{-1}p_t^m\bigr)\dif t
				+\gamma^{-1}\sigma\dif Z_t.
			\end{cases}
		\end{align}
		This change of variables removes the negative power of $m$ from the
		noise coefficient while keeping the fast dissipation in the momentum
		equation. In the new coordinates $u=(\omega,p)\in\R^{2d}$, write
		\[
		Q(\omega,p)=|\omega|^2+|p|^2=|u|^2,
		\qquad W(\omega,p)=(1+Q(\omega,p))^{\beta/2}.
		\]
		Let $G_m$ be the generator of \eqref{E:ScaledSys}. For suitable $g$,
		\begin{align*}
			G_mg(u)&=D_mg(u)+J_mg(u),\\
			D_mg(u)&=\left\langle \frac{\partial}{\partial \omega} g(u),\gamma^{-1}b(\omega-p)\right\rangle
			+\left\langle \frac{\partial}{\partial p} g(u),\gamma^{-1}b(\omega-p)-\gamma m^{-1}p\right\rangle,\\
			J_mg(u)&=\int_{\R_0^d}\bigl[
			g(u+w_m(z))-g(u)-\langle w_m(z),\nabla g(u)\rangle
			1_{\{|w_m(z)|\leq1\}}\bigr]\nu(\dif z),
		\end{align*}
		where
		\[
		w_m(z)=\gamma^{-1}
		\begin{pmatrix}\sigma z\\\sigma z\end{pmatrix}\in\R^{2d},
		\qquad z\in\R^d.
		\]

		\noindent\textbf{Drift part $D_mW$.}
		By dissipativity and Young's inequality,
		\begin{align}\label{e:drift-coercivity}
			\notag\langle\omega,b(\omega)\rangle
			&\leq-c_1|\omega|^2+c_2+\langle\omega,b(0)\rangle\\
			&\leq-\frac{c_1}{2}|\omega|^2+c_2+\frac{|b(0)|^2}{2c_1}
			=:-a|\omega|^2+B.
		\end{align}
		Writing $K$ for the Lipschitz constant of $b$, we obtain
		\[
		\langle\omega+p,b(\omega-p)\rangle
		\leq-a|\omega|^2+B+2K|\omega||p|+K|p|^2+|b(0)||p|.
		\]
		We apply the drift operator $D_m$ to the quadratic function $Q$ to obtain
		\begin{align*}
			D_mQ(\omega,p)
			&=\frac2\gamma\langle\omega+p,b(\omega-p)\rangle
			-\frac{2\gamma}{m}|p|^2\\
			&\leq-\frac{2a}{\gamma}|\omega|^2+\frac{2B}{\gamma}
			+\frac{4K}{\gamma}|\omega||p|
			+\frac{2K}{\gamma}|p|^2
			+\frac{2|b(0)|}{\gamma}|p|-\frac{2\gamma}{m}|p|^2.
		\end{align*}
		Using
		\[
		4K|\omega||p|\leq a|\omega|^2+\frac{4K^2}{a}|p|^2,
		\qquad 2|b(0)||p|\leq|p|^2+|b(0)|^2,
		\]
		we get
		\[
		D_mQ(\omega,p)\leq-\kappa|\omega|^2
		+\left(C_1-\frac{2\gamma}{m}\right)|p|^2+C_0,
		\]
		where
		\[
		\kappa=\frac a\gamma,\qquad
		C_1=\frac{4K^2/a+2K+1}{\gamma},\qquad
		C_0=\frac{2B+|b(0)|^2}{\gamma}.
		\]
		Choose
		\begin{align*}
			m_0=\min\left\{1,\frac{2\gamma}{C_1+\kappa}\right\},
		\end{align*}
        which is the same as $m_0$ in \eqref{E:small-mass-range}. For $0<m\leq m_0$, it follows that
		$D_mQ\leq-\kappa Q+C_0$. Since $\beta<\alpha<2$,
		\begin{align*}
			D_mW(\omega,p)
			&=\frac\beta2(1+Q)^{\beta/2-1}D_mQ\\
			&\leq-\frac{\beta\kappa}{2}(1+Q)^{\beta/2}
			+\frac\beta2(\kappa+C_0)(1+Q)^{\beta/2-1}\\
			&\leq-c_\beta W(\omega,p)+C_\beta,
		\end{align*}
		where the constant  $c_\beta=\beta\kappa/2>0$.

		\noindent\textbf{Jump part $J_mW$.}
		Symmetry of the L\'evy measure gives
		\begin{equation}\label{E:JumpSymmetry}
			J_mW(u)=\frac12\int_{\R_0^d}
			\bigl[W(u+w_m(z))+W(u-w_m(z))-2W(u)\bigr]\nu(\dif z).
		\end{equation}
		Here $|w_m(z)|\leq C|z|$ uniformly for $m\leq m_0$.
		We claim that
		\begin{equation}\label{E:SecondDiff_W}
			|W(u+w)+W(u-w)-2W(u)|\leq C_\beta
			\begin{cases}
				|w|^2,&|w|\leq1,\\
				|w|^\beta,&|w|>1,
			\end{cases}
		\end{equation}
		for all $u,w\in\R^{2d}$. Indeed,
		\[
		\|\nabla^2W(u)\|\leq C_\beta(1+|u|^2)^{(\beta-2)/2}
		\leq C_\beta.
		\]
		For $|w|\leq1$, the estimate follows from the bounded Hessian and
		Taylor's formula. For $|w|>1$ and $|u|\leq2|w|$, the growth of $W$ gives
		\[
		|W(u+w)+W(u-w)-2W(u)|
		\leq W(u+w)+W(u-w)+2W(u)\leq C_\beta|w|^\beta.
		\]
		If $|w|>1$ and $|u|>2|w|$, then $|u\pm sw|\geq|u|/2$ for $s\in[0,1]$.
		Taylor's formula therefore gives
		\begin{align*}
			&|W(u+w)+W(u-w)-2W(u)|\\
			&\quad\leq |w|^2\int_0^1(1-s)
			\bigl(\|\nabla^2W(u+sw)\|+\|\nabla^2W(u-sw)\|\bigr)\dif s\\
			&\quad\leq C_\beta|w|^2|u|^{\beta-2}
			\leq C_\beta|w|^\beta.
		\end{align*}
		This proves \eqref{E:SecondDiff_W}.
		Consequently,
		\begin{align*}
			|J_mW(u)|
			&\leq C_\beta\int_{|z|\leq1}|z|^2\nu(\dif z)
			+C_\beta\int_{|z|>1}|z|^\beta\nu(\dif z)\\
			&\leq C_\beta\left(\int_0^1 s^{1-\alpha}\dif s
			+\int_1^\infty s^{\beta-\alpha-1}\dif s\right)
			\leq C_\beta.
		\end{align*}
		Combining the drift and jump estimates yields
		$G_mW\leq-c_\beta W+C_\beta$, or equivalently,
		\begin{equation}\label{E:Lyapunov condition}
			\mathcal A_mV_m(x,y)\leq-c_\beta V_m(x,y)+C_\beta.
		\end{equation}
		
	Using It\^{o}'s formula and taking expectation with respective to ergodic measure $\mu_m$ (see Proposition \ref{org:kinetic-ergodicity}), that is, $(X_0^m,Y_0^m) \sim \mu_m$, by stationarity, 
			\begin{align*}
				\E_{\mu_m} V_m(X_t^m,Y_t^m) &= \mu_m(V_m), \forall t\geq 0. 
			\end{align*}
			Thus, one knows that
			\begin{align}\label{E:Vm-time-bound}
				\E_{\mu_m} V_m(X_t^{m},Y_t^{m})
				\leq e^{-c_\beta t} \E_{\mu_m} V_m(x,y)
				+\frac{C_\beta}{c_\beta}(1-e^{-c_\beta t}), \forall t\geq 0.
			\end{align}
			Taking $t=1$, it implies that
			\begin{align*}
				\mu_m(V_m) &\leq C.
			\end{align*}
			Combining with the definition of $V_m$, one knows that 
			\begin{align*}
				\E_{\mu_m}|X_t^{m}|^\beta
				&\leq C,
			\end{align*}
			which implies \eqref{Xm moment estimate}.
		
			\underline{(ii) Velocity moments.} Let the initial state $(X_0^m,Y_0^m) \sim \mu_m$. We know the process $(X_t^m,Y_t^m)_{t\geq 0}$ is stationary and $\mathcal{L}(X_t^m,Y_t^m) \overset{d}{=} \mu_m$ for all $t\ge 0$, variation of constants gives
			\begin{align}\label{E:velocity-variation}
				Y_t^{m}
				&=e^{-\gamma t/m} Y_0^m
				+m^{-1}\int_0^t e^{-\gamma(t-s)/m}b(X_s^{m})\dif s
				+m^{-1}\int_0^t e^{-\gamma(t-s)/m}\sigma\dif Z_s  \nonumber \\
				& \eqqcolon e^{-\gamma t/m}Y_0^m+B_t^m+S_t^m.
			\end{align}
			
			Due to $Y_t^m \overset{d}{=} Y_0^m$, by setting $t=m_0\gamma^{-1}\log 2$ in \eqref{E:velocity-variation}, one knows that
			\begin{align}\label{E|Y|}
				(1-e^{-\gamma t/m}) \E_{\mu_m} |Y_t^m| &\geq \frac{1}{2}  \E_{\mu_m} |Y_t^m|.
			\end{align}
			
			By the linear growth of $b$ and the position moment estimate \eqref{Xm moment estimate} which is established in the last step, we know that 
			\begin{align*}
				\sup_{0\leq s\leq t} \E_{\mu_m}|b(X_s^{m})|^\beta\leq C_\beta.
			\end{align*}
			The nonnegative kernel
			$k_{m,t}(s)=m^{-1}e^{-\gamma(t-s)/m}$ has integral
			$\gamma^{-1}(1-e^{-\gamma t/m})\leq\gamma^{-1}$.
			Weighted Jensen's inequality gives
			\begin{align}\label{E:velocity-drift-bound}
				\E|B_t^m|^\beta
				&\leq\left(\int_0^t k_{m,t}(s)\dif s\right)^{\beta-1}
				\int_0^t k_{m,t}(s)\E|b(X_s^{m})|^\beta\dif s
				\leq C_\beta.
			\end{align}

			For $\beta\in (0,\alpha)$, it follows from \cite[Section 4.3.4]{Applebaum2009levy} (or \cite[Section 3.4]{Samorodnitsky1994stable}) that
			\begin{align}\label{E:velocity-noise-bound}
				\E\left|S_t^m \right|^\beta 
				&\leq C\left(
				\int_0^t m^{-\alpha} e^{-\alpha \gamma(t-s)/m} \dif s
				\right)^{\beta/\alpha} 
				\leq Cm^{-\beta(1-1/\alpha)}.
			\end{align}
			Using $|a+b+c|^\beta\leq C_\beta(|a|^\beta+|b|^\beta+|c|^\beta)$
			in \eqref{E:velocity-variation}, it follows from \eqref{E|Y|}-\eqref{E:velocity-noise-bound} that for $1\leq\beta<\alpha$,
			\begin{align*}
				\E_{\mu_m}
				|Y_t^{m}|^{\beta}
				&\leq C(1-e^{-\gamma t/m}) \E_{\mu_m}
				|Y_t^{m}|^{\beta} 
				\leq 
				C\E_{\mu_m} |B_t^m|^{\beta}
				+C\E_{\mu_m} |S_t^m|^{\beta}
				\leq C(1+m^{-\beta(1-1/\alpha)}).
			\end{align*}
			The proof of Lemma \ref{lem:moment estimates} is complete.
	\end{proof}

	\section{Stein equation}\label{org:stein}
    The proof of Proposition \ref{lem:newReg} relies on the following regularity estimates of the solution $f_h$ to the Stein's equation \eqref{stein equ}. The proof of these estimates are deferred to the end of this section.
	\begin{lemma}\label{pro:stein regularies}
		Under Assumption \ref{Assumption b}, for $h\in {\rm Lip}(1)$, let $f_h$ be the solution for Stein's equation \eqref{stein equ}, there exists constant $C>0$ independent of $x$  such that
		\begin{align*}
			|f_h(x)| \leq C(1+|x|), \quad 
			\| \nabla f_h \|_{\infty} \leq C.
		\end{align*}
	\end{lemma}

	\begin{proof}[Proof of Proposition \ref{lem:newReg}]
	(i) For $h\in {\rm Lip}(1)$, it follows from Lemma \ref{pro:stein regularies} that the solution of Stein's equation $f_h\in \mathcal{C}^1(\R^d)$ and one knows that
	\begin{align}\label{e:graf}
	\nabla f_{h,\varepsilon}(x)
	&=\int_{\R^d}\nabla f_{h}(x-z)\rho_\varepsilon(z)\dif z,	
	\end{align}
	which implies the desired estimate for $\| \nabla f_\varepsilon \|_{\infty}.$
	
	(ii) Since
	$\nabla\rho_\varepsilon(z)
	=\varepsilon^{-d-1}\nabla\rho(z/\varepsilon)$, the substitution $z=\varepsilon w$ gives
	\begin{align}\label{e:rhoL1}
	\|\nabla\rho_\varepsilon\|_{L^1}
	&=\varepsilon^{-d-1}
	\int_{\R^d}|\nabla\rho(\varepsilon^{-1}z)|\dif z
	=\varepsilon^{-1}
	\int_{\mathbb R^d}|\nabla\rho(w)|\dif w
	\lesssim \e^{-1}.
	\end{align}

	It follows from \eqref{e:graf} that
	\begin{align*}
	\nabla f_{h,\varepsilon}(x)
	=\int_{\R^d}\nabla f_{h}(y)\rho_\varepsilon(x-y)\dif y.
	\end{align*}
	For unit vectors \(u,v\in\mathbb R^d\), denoting by $D_u$ and $D_v$ their directional derivatives respectively, one obtains that
	\begin{align*}
	D_v D_u f_{h,\varepsilon}(x)
	&=D_v\int_{\R^d}
	\langle\nabla f_{h}(y),u\rangle\rho_\varepsilon(x-y) \dif y
	=\int_{\R^d}
	\langle\nabla f_{h}(y),u\rangle
	\langle\nabla\rho_\varepsilon(x-y),v\rangle\dif y,
	\end{align*}
	which implies that
	\begin{align*}
	|D_vD_uf_\varepsilon(x)|
	&\leq \|\nabla f_{h,\e}\|_\infty
	\int_{\R^d}|\nabla\rho_\varepsilon(x-y)|\dif y 
	=\|\nabla f_{h}\|_\infty\|\nabla\rho_\varepsilon\|_{L^1}.
	\end{align*}
	It follows from \eqref{e:rhoL1} and Lemma \ref{pro:stein regularies} that taking the supremum over $x$ and unit $u,v$ yields
	\begin{align*}
	\|\nabla^2f_{h,\varepsilon}\|_\infty
	\le\|\nabla f_{h}\|_\infty\|\nabla\rho_\varepsilon\|_{L^1}
	\lesssim \e^{-1}.
	\end{align*}
	
	(iii) Since $\rho_\varepsilon$ has integral one, for every $x\in\R^d$, one gets that
	\begin{align*}
	h(x)-h_\varepsilon(x)
	&=\int_{\R^d}
	[h(x)-h(x-z)]\rho_\varepsilon(z)\dif z.
	\end{align*}
	By the fact $h\in {\rm Lip}(1)$ and the change of variables $z=\varepsilon w$, one knows that
	\begin{align*}
	|h(x)-h_\varepsilon(x)|
	&\leq \int_{\R^d}
	|h(x)-h(x-z)|\rho_\varepsilon(z)\dif z
	\leq \int_{\R^d} |z|\rho_\varepsilon(z)\dif z \\
	&=\int_{\mathbb R^d}
	|z|\varepsilon^{-d}\rho(z/\varepsilon)\dif z =\varepsilon\int_{\mathbb R^d}|w|\rho(w)\dif w 
	\lesssim \varepsilon.
	\end{align*}

	(iv) For each $u\in\R^d$, since $b$ is Lipschitz and estimate for $\| \nabla f_{h} \|_{\infty}$ in Lemma \ref{pro:stein regularies}, one knows that
	\begin{align*}
		|R_{h,\varepsilon}(u)|
		&\leq \gamma^{-1}\int_{\R^d}
		|b(u)-b(u-z)| |\nabla f_h(u-z)|
		\rho_\varepsilon(z)\dif z \\
		&\lesssim \| \nabla f_{h} \|_{\infty} \int_{\R^d} |z| \rho_\varepsilon(z) \dif z 
		\lesssim \e.
	\end{align*}

	(v) For any fixed $\e\in (0,1)$, one knows that $\| \nabla f_{h,\e} \|_{\infty} \leq C$ and $ 
	\| \nabla^2 f_{h,\e} \|_{\infty} \leq C\e^{-1}$, which implies that
	\begin{align*}
	& (-\Delta)^{\alpha/2}_{m^{-1}\sigma,2}\Phi_{m,h,\e}(x,y)
	=\int_{|m^{-1}\sigma z|> 1} \left[\Phi_{m,h,\e}\left(x,y+m^{-1}\sigma z\right) -\Phi_{m,h,\e}(x,y) \right] \frac{d_{\alpha}}{|z|^{\alpha+d}} \dif z \\
	&+\int_{|m^{-1}\sigma z|\leq 1} \left[\Phi_{m,h,\e} \left(x,y+m^{-1}\sigma z\right) -\Phi_{m,h,\e}(x,y)- \langle m^{-1}\sigma z, \frac{\partial}{\partial y} \Phi_{m,h,\e}(x,y) \rangle \right] \frac{d_{\alpha}}{|z|^{\alpha+d}} \dif z
	\end{align*}
	is well defined and we can get the following estimate
	\begin{align*}
	|\mathcal{A}_m \Phi_{m,h,\varepsilon}(x,y)|
	\leq C(1+|x|+\varepsilon^{-1}).
	\end{align*}
	
	(vi) Recall that $\mu_m$ is the unique invariant measure so that $(X_t^m,Y_t^m)_{t\ge 0}$ is stationary with the initial pair $(X_0^m,Y_0^m)\sim \mu_m$. For each fixed $\varepsilon>0$, we know that $$\E_{\mu_m} \mathcal{A}_m \Phi_{m,h,\varepsilon}(X_t^m,Y_t^m)\equiv c_0$$ for certain constant $c_0\in \R$ and for all $t\ge 0$. By Dynkin's formula and taking expectation with respective to ergodic measure $\mu_m$, one yields that
	\begin{align*}
	0&=\E_{\mu_m} f_{h,\varepsilon}(\Omega_t^m)-\E_{\mu_m} f_{h,\varepsilon}(\Omega_0^m)
	=\int_0^t \E_{\mu_m} \mathcal{A}_m \Phi_{m,h,\varepsilon}(X_s^m,Y_s^m)\dif s
    =c_0 t, \quad \forall  t\ge0.
	\end{align*}
	Thus, we know that
	\begin{align*}
	\E_{\mu_m} \mathcal{A}_m \Phi_{m,h,\varepsilon}(X_t^m,Y_t^m)
    \equiv c_0
    =0, \quad \forall  t\ge 0.
	\end{align*}
	The proof of Proposition \ref{lem:newReg} is complete.
	\end{proof}

	\begin{proof}[Proof of Lemma \ref{pro:stein regularies}.]
		The solution of Stein equation in \eqref{stein equ} can be solved as 
		\begin{align}\label{solution f_h}
			f_h(x) &= -\int_0^{\infty} (\E h(X_t^x) - \mu(h) ) \dif t
			=-\int_0^{\infty} (P_t h(x) - \mu(h) ) \dif t,
		\end{align}
		where $P_t$ is the semigroup of process $(X_t)_{t\geq 0}$.

		(i) The regularities for $f_h$ can be proven by the exponential ergodicity for process $(X_t)_{t\geq 0}$ in Proposition \ref{pro ergodic X}, then there exist some positive constants $c$ and $C$ such that
		\begin{align*}
			|f_h(x)| &\leq \int_0^{\infty} |P_t h(x) - \mu(h)| \dif t
			\leq \int_0^{\infty} Ce^{-ct} (1+|x|) \dif t
			\leq C(1+|x|).
		\end{align*}
		
		(ii) It follows from \cite[Theorem 1.3]{Liang2020Gradient} that there exist some positive constants $c$ and $C$ such that
		\begin{align*}
			|P_t h(x)-P_t h(y)| &\leq \mathcal{W}_1( \mathcal{L}(X_t^x), \mathcal{L}(X_t^y) )
			\leq C e^{-c t}|x-y|,
		\end{align*}
		which implies that 
		\begin{align*}
			|\nabla P_t h(x)| &\leq Ce^{-ct}.
		\end{align*}
		Combining this with \eqref{solution f_h}, one knows that there exists some positive constant $C$ independent of $x$ and $t$ such that
		\begin{align*}
			|\nabla f_h(x)| &\leq \int_0^{\infty} |\nabla P_t h(x)| \dif t 
			\leq C.
		\end{align*}
		The proof of Lemma \ref{pro:stein regularies} is complete.
	\end{proof}

	

\end{document}